\documentclass[12pt]{article}
\usepackage{amsmath, amsthm, amssymb, fullpage}
\usepackage{parskip}
\usepackage{cite}
\usepackage[utf8]{inputenc}
\usepackage[english]{babel}
\usepackage{eucal}
\usepackage{hyperref}
\usepackage{cleveref}
\usepackage{indentfirst}
\usepackage{titlesec}
\usepackage{mathtools}
\usepackage{comment}
\usepackage{enumitem}

\newcommand{\mtn}{M_{thin}}

\newcommand{\vol}{\textnormal{vol}}
\newcommand{\wm}{\widetilde{M}}

\newtheorem{theorem}{Theorem}[section]
\newtheorem{corollary}{Corollary}[section]
\newtheorem{lemma}{Lemma}[section]
\newtheorem{proposition}{Proposition}[section]

\newtheorem{remark}{Remark}[section]

\title{On Eigenvalues of Geometrically Finite Hyperbolic Manifolds of Infinite Volume}
\author{Xiaolong Hans Han}
\date{}

\begin{document}

\maketitle

\begin{abstract}
	Let $M$ be an oriented geometrically finite hyperbolic manifold of infinite volume with dimension $n \geq 3$. For all $k \geq 0$, we provide a lower bound on the $k$-th eigenvalue of the Laplace–Beltrami operator of $M$ by the $k$-th eigenvalue of some neighborhood of the thick part of the convex core, up to a constant. As an application, we recover a theorem similar to the one of M. Burger and R. Canary which bounds the bottom $\lambda_0$ of the spectrum from below by $\frac{c}{\vol(C_1(M))^2}$, where $C_1(M)$ is the $1$-neighborhood of the convex core and $c$ is a constant.  
\end{abstract}

\section{Introduction} 
As analytic data of a manifold, the spectrum of the Laplace-Beltrami operator contains rich information about the geometry and topology of a manifold. For example, by results of J. Cheeger \cite{cjLowerBoundSmallest} and P. Buser \cite{bpNoteIsoperimetric}, on a compact manifold with a lower bound on the Ricci curvature, the first nonzero eigenvalue can be bounded from below and above by quadratic expressions on the Cheeger constant. The rigidity of hyperbolic geometry makes the connection even more intriguing: according to Canary\cite{canary1992}, for infinite-volume, topologically tame hyperbolic 3-manifolds, the bottom of the $L^2$-spectrum of  $-\Delta$, denoted by $\lambda_0$, is zero if and only if the manifold is \textit{not} geometrically finite. If $\lambda_0$ of a geometrically finite, infinite-volume hyperbolic $3$-manifold $M$ is $1$, then Canary and E. Taylor \cite{ctKleinianGroupsLimit} classify its topology: $M$ is either the topological interior of a handlebody or an $\mathbb{R}$-bundle over a closed surface. \par

On the other hand, geometric invariants of a manifold, such as volume, pose restrictions on its spectrum. For closed hyperbolic manifolds $M$, R. Schoen \cite{schoen1982} shows that the first nonzero eigenvalue is bounded from below by $\frac{c}{\vol(M)^2}$, where $c$ is a universal constant. N. White\cite{doi:10.1112/jlms/jds082} show that for $\epsilon$-thick closed hyperbolic $3$-manifolds with the rank of their fundamental groups $\leq m$, the $k$-th eigenvalue $\lambda_k$ is approximately $\frac{c}{\vol(M)^2}$, where $c$ is a constant depending on $\epsilon$, $m$ and $k$. For noncompact negatively curved manifolds of finite volume, B. Randol \cite{dodziuk1986lower} and J. Dodziuk\cite{dod1987} show that the bottom of the spectrum has a similar lower bound by $\frac{c}{\vol(M)^2}$. In a series of papers, P. Lax and R. Phillips \cite{LAX1982280,lpI,lpII,lpIII}, and P. Hislop [\citenum{Hislop1994}] have studied the spectrum of Laplacian on geometrically finite manifolds and establish the subtle issue of the existence of discrete eigenvalues of noncompact geometrically finite hyperbolic $n$-manifolds: there can be only finitely many eigenvalues of finite multiplicity in $[0,(n - 1)^2/4)$. There are finite-volume examples with infinitely many eigenvalues in $[(n - 1)^2/4, \infty)$, but they seem to occur quite rarely. On the contrary, geometrically finite manifolds of infinite volume have no eigenvalues in $[((n - 1)/2)^2, \infty)$. J. Lott \cite{lott1997l2} and R. Canary \cite{canary1992} have obtained interesting results on eigenvalues of geometrically infinite manifolds. Hamenstädt \cite{hamenstadt2018small} relates the eigenvalues of an oriented noncompact finite-volume pinched negatively curved $n$-manifold to the eigenvalues of the compact, thick part of the manifold. In particular, for every oriented finite-volume Riemannian manifold $M$ of dimension $n \geq 3$ and sectional curvature $ \kappa \in [-b^2, -1]$ and for all $k \geq 0$, Hamenstädt proved that 

\begin{center}
	$\Large \lambda_k(M) \geq $ min \{$\frac{1}{3} \lambda_k(\widehat{M}_{thick}) ,\frac{(n-2)^2}{12}$\},
\end{center} 
where $\widehat{M}_{thick}$ denotes the thick part of $M$ under a suitable choice of Margulis constant, and $\lambda_k(M)$ denotes the $k$-th discrete eigenvalue of the Laplace-Beltrami operator of $M$ counting with multiplicities. \par
Let $M$ be a geometrically finite, hyperbolic $n$-manifold of infinite volume. Denote by $\widetilde{M}$ the neighborhood of the thick part of the convex core of $M$, constructed in Section 2 \ref{construCore} Based on the techniques of \cite{hamenstadt2018small}, we bound the eigenvalues of the Laplace-Beltrami operator of $M$ by those of $\widetilde{M}$: 

\begin{theorem}\label{mainth}
	Let $M$ be an oriented geometrically finite hyperbolic manifold of infinite volume with dimension $n \geq 3$, whose cusps are all of maximal rank. Denote by $\lambda_k(M)$ the $k$-th discrete eigenvalue of the Laplace-Beltrami operator of $M$. Then for all $k \geq 0$, we have 
	\begin{equation}
		\lambda_k(M) \geq  \min \{\frac{1}{3} \lambda_k(\widetilde{M}) ,\frac{(n-2)^2}{12}\},
	\end{equation} 
	where the boundary condition on $\widetilde{M}$ is Neumann.  
\end{theorem}

\begin{remark}
	The number $\frac{(n-2)^2}{12}$ comes with the following intuition. The constant $\frac{(n-2)^2}{4}$ is closely related to the geometry of the Margulis tubes and cusps: it is the infimum of Rayleigh quotients for tubes and finite-volume cusps, which follow from the metric. For the class of functions we are interested in, the $L^2$-norm of $f$ on $\widetilde{M}$ is at least $1/3$ of the $L^2$-norm of $f$ on $M$. Such a lower bound can be summarized as follows: for functions with small enough Rayleigh quotients, the mass is at least $\frac{1}{3}$-concentrated on $\widetilde{M}$. 
\end{remark}
A natural corollary, using the uniform geometry of the core of a geometrically finite manifold and comparison with the eigenvalues of a uniform graph, is the following.
\begin{corollary}
	There exists a constant $C(n)$ that only depends on the dimension such that 
	\begin{equation}
		\lambda_0(M) \geq \frac{C(n)}{\vol(\widetilde{M})^2}. 
	\end{equation}
\end{corollary}
\begin{remark}
	This generalizes [\citenum{hamenstadt2018small}, Corollary 2.6]. Compare with the main theorem of \cite{burger1994lower}, which has a similar lower bound on $\lambda_0(M)$ using the volume of the $1$-neighborhood of the convex core replacing $\vol(\widetilde{M})$. Our construction of $\widetilde{M}$ is more complicated than taking a neighborhood of the convex core, but the proof is conceptually more accessible and can be adapted to more general settings such as pinched negative curvatures. 
\end{remark}

\section{Preliminaries}\label{defBackground}
Let $M$ be an oriented geometrically finite hyperbolic manifold of infinite volume with  
dimension $n\geq3$, whose cusps are of maximal rank. This section introduces some important definitions and theorems. We also describe the construction of the core $\widetilde{M}$, building on properties of geometric finiteness. Readers may refer to \cite{rjFoundationHypMan} for general background on (geometrically finite) hyperbolic manifolds. \par

\begin{enumerate}[label=\Roman*.]
	\item Rayleigh Quotient: \par
	For a smooth square-integrable nonzero function $f$ on $M$, the Rayleigh quotient of $f$, denoted by $\mathcal{R}(f)$, is \par
	\begin{center}
		$\mathcal{R}(f)=\displaystyle\frac{\int_M \| \nabla f \|^2}{\int_M f^2}$.
	\end{center}
	
	\item Min-max Theorem for Self-Adjoint Operators\par
	The Laplace-Beltrami operator of a noncompact geometrically finite manifold is self-adjoint (see \cite{Hislop1994}). Recall that the discrete eigenvalues are less than $\frac{(n-1)^2}{4}$, which is the bottom of the essential spectrum (see [\citenum{Hislop1994}, Theorem 2.12]). The following Min-max theorem \cite{teschlmathematical} is used later: \par 
	\begin{theorem}
		Let $A$ be a self-adjoint operator, and let ${\lambda_{0}\leq \lambda_{1}\leq \lambda_{2}\leq \lambda_{3}\leq \cdots } $  be the eigenvalues of $A$ below the essential spectrum $\sigma _{ess}(A)$. Then \par
		\begin{center}
			${\lambda_{n}=\min \limits_{\psi _{1},\ldots ,\psi _{n}}\max\{\langle \psi ,A\psi \rangle :\psi \in {span} (\psi _{1},\ldots ,\psi _{n}),\,\|\psi \|=1\}}$,
		\end{center}
		where $\psi _i$'s lie in the domain of $A$ and are linearly independent.
		If there are only $N$ eigenvalues counting with multiplicities, then we let ${\displaystyle \lambda_{n}:=\inf \sigma _{ess}(A)}$ (the bottom of the essential spectrum) for $n > N$, and the above statement holds after replacing min-max with inf-sup. 
	\end{theorem}
	
	\item Thick-thin Decomposition for Geometrically Finite Manifolds of Infinite Volume and the Core $\widetilde{M}$:\par
	For completeness, the descriptions of cusps and Margulis tubes from \cite{hamenstadt2004small, hamenstadt2018small} are included and then modified, building on properties of geometric finiteness.
	\begin{enumerate}
		\item Convex core, a neighborhood of an end of infinite-volume and geometrically finite manifolds: \par
		The convex core is the smallest closed convex submanifold $C(M)$ such that the inclusion map $C(M)\hookrightarrow M$ is a homotopy equivalence. A hyperbolic manifold $M$ is called geometrically finite if the volume of any $r$-neighborhood of the convex core $N(C(M),r)$ is finite ([\citenum{rjFoundationHypMan}, Theorem 12.7.2]). The complement of the convex core, $M-C(M)$ is an open subset of $M$ with finitely many components. Each component is a neighborhood of an end of infinite volume. The boundary of the $r$-neighborhood of the convex core $N(C(M),r)$ is compact. Thus we can perturb the boundary slightly so that it is smooth. Denote such an $r$-neighborhood with smooth perturbed boundary by $\widetilde{N}(C(M),r)$. Each component of the complement of $\widetilde{N}(C(M),r)$ can be parameterized by $\mathbb{R}^+ \times \Omega$, where $\Omega$ is hypersurface. When the dimension of $M$ is $3$, the hypersurfaces consist of surfaces of genera $\geq 2$.  
		
		\item Cusps: \label{cusp} \par
		A (truncated) cusp $C$ is the quotient of a horoball under the action of a parabolic subgroup $\Gamma$. A cusp of maximal rank corresponds to a parabolic subgroup of rank $n-1$, which also corresponds to an end of $M$ of finite volume. The metric structure on a cusp $C$ is as follows: The subgroup $\Gamma$ stabilizes a flat horosphere $H$ in the universal cover $X$ of $M$. The cusp $C$ is diffeomorphic to $H/ \Gamma \times [0, \infty)$ with the diffeomorphism mapping each ray ${z} \times [0,\infty)$ to a geodesic in $C$. Denote by $dx^2$ the induced metric on $H/ \Gamma$ of the Riemannian metric on $X$; then the metric $g$ on $C$ can be written as $g = e^{-(n-1)t}dx^2 + dt^2$. Note that the metric $dx^2$ is homothetic to a Euclidean metric ($dx^2=c \cdot ds^2_{Eucli}$). A geometrically finite manifold has finitely many cusps [\citenum{rjFoundationHypMan}, Theorem 12.7.2]. By making each cusp smaller if necessary, we assume it is embedded in $M$ and two distinct cusps are disjoint. 
		
		\item Margulis tubes and the thick-thin decomposition: \label{margulisTube} \par
		The Margulis constant $\epsilon=\epsilon(n)>0$ of dimension $n$ is the largest constant such that the thin part $M_{thin} \coloneqq \{x\in M | inj(x)< \epsilon/2 \}$ consists of cusps and embedded tubular neighborhoods of closed geodesics of length $\ell < \epsilon$. Such a tubular neighborhood $T$ is called a Margulis tube, and its closed geodesic $\gamma$ is called the core geodesic. Topologically the tube $T$ is $S^1 \times D^{n-1}$, where $S^1$ corresponds to $\gamma$ and $D^{n-1}$ is an $(n-1)$-dimensional open disc. This disc parameterizes the normal bundle of $\gamma$. We use Fermi coordinates adapted to $\gamma$ for the metric and various computations. We start by fixing a parameterization of $\gamma$ by arc length on the interval $[0, \ell)$, whose coordinate is denoted by $s$. We then use polar coordinates on $D\coloneqq D^{n-1}$. Let $\rho \geq 0$ be the radial distance from $\gamma$, which is also the radial coordinate on $D$. A point $\gamma(s)$ on the geodesic corresponds to $\rho=0$ and the center of $s \times D$. Let $\sigma$ be the standard angular coordinates on $s\times D$ obtained by parallel transport of the angular coordinate on $\gamma(0)\times D$. The unit normal bundle on $D$ is naturally isomorphic to $S^{n-2}$. Via the normal exponential map, these functions define coordinates $(\sigma, s, \rho)$ on the complement of the core geodesic in the tube, $T -{\gamma}$, whose domain is $S^{n-2} \times [0, \ell) \times(0,\infty)$. Points on the geodesic $\{\rho=0\}$ cannot be assigned an angular coordinate. In these coordinates, the maps $\rho \rightarrow (\sigma, s, \rho)$ are unit speed geodesics with starting points on $\gamma$ and initial velocity perpendicular to $\gamma '(s)$. Denote $N(\gamma)$ the normal bundle of $\gamma$. There exists a continuous function $R : N(\gamma) \rightarrow (0,\infty), (\sigma, s) \rightarrow R(\sigma, s)$ such that in these coordinates, we have $T = \{\rho \leq R_T\}$. The maximal constant $R_T$ such that $T$ is embedded is called the radius of tube $T$. The metric on $T - {\gamma}$ is of the form $h(\rho) + d\rho^2$ where $h(\rho)$ is a family of metrics on the hypersurfaces $\rho = const$. \par 
		A geometrically finite manifold has finitely many tubes [\citenum{rjFoundationHypMan}, Theorem 12.7.8]. Thus the thin part $\mtn$ consists of finitely many components.   
		The thick part $M_{thick}$ is defined as $\{x\in M | inj(x) \geq \epsilon/2 \}$. The decomposition $M=M_{thin}\cup M_{thick}$ is called the thick-thin decomposition. By [\citenum{buser1993tubes}, Lemma 2.4], if a tube has radius $\leq 1$, then the length of its core geodesic is longer than a constant depending only on the dimension. Thus up to picking a slightly smaller Margulis constant, which depends only on the dimension, we can slightly adjust the thick-thin decomposition and replace $M_{thick}$ by its union with all Margulis tubes $\{ T: R_T \leq  1 \}$. Henceforward we assume the radius of each Margulis tube is uniformly bounded below by $1$. \par  
		
	\end{enumerate}
	
	\item Shell Estimates:\par
	[\citenum{dodziuk1986lower}, P134] first defined shells for finite-volume hyperbolic manifolds. We adapt their definitions for general geometrically finite hyperbolic manifolds. A shell for a tube $T$ is the subset of $T$ for which $R_T-1 \leq \rho \leq R_T$. A shell for a cusp $T=H/ \Gamma \times [0, \infty)$ is a subset described by $0 \leq t \leq 1$. A shell for an end is a component of $\widetilde{N}(C(M),r+1)-\widetilde{N}(C(M),r)$. Each shell is thus compact. There are only finitely many shells by properties of geometric finiteness. Shell estimates refer to a set of prescribed conditions on functions on shells so that useful conclusions can be drawn, such as inequalities involving functions or their gradient. Shell estimates appeared in the literature of discrete spectrum of negatively curved manifolds for the first time in \cite{dodziuk1986lower} as Lemma 2, which we record here for convenience. It is a prototype of an argument we will use several times later.  \par
	\begin{lemma}
		For each $n$, there exists a constant $\delta > 0$ such that if $T$ is a thin component of $M$ with shell $S$, and if $f$ is a function defined on $T$ and satisfying:
		
		\begin{enumerate}
			\item $\int_T |f|^2 = c> 0$,
			\item $\int_S \|\nabla f\|^2 < \delta c$,
			\item $\int_S |f|^2 < \delta c$,
		\end{enumerate} \par
		then $\int_T \|\nabla f\|^2 > (c/2)((n - 1)/2)^2$.
	\end{lemma} \par
	This lemma is used in \cite{dodziuk1986lower} and \cite{burger1994lower} to provide a lower bound for the first eigenvalue in the hyperbolic case. Shell estimates are also used in \cite{hamenstadt2018small} in a critical way to give a converse inequality to \Cref{mainth} for hyperbolic $3$-manifold with finite volume. 
	
	\item Construction of the Core $\widetilde{M}$: \label{construCore}\par
	We start with the $r$-neighborhood of the convex core $N(C(M),r)$, where $r$ is a fixed constant such that $r \geq \tanh^{-1} \frac{(n-2)^2}{(n-1)^2}$. The lower bound on $r$ ensures that the inequality of \Cref{pr3.4} holds. Note that $N(C(M),r)$ has finite volume and contains all the tubes and cusps. We then remove all the cusps and Margulis tubes from $N(C(M),r)$. Denote it by $\widehat{M}$. In dimension $\geq 3$, $\widehat{M}$ is connected. In the next step, we glue $\widehat{M}$ with all the shells of tubes, cusps, and ends, obtaining a compact manifold with boundary. Since $M$ is geometrically finite, only finitely many shells are attached. In the last step, we perturb the boundary slightly so that it becomes smooth. Denote it by $\widetilde{M}$ and call it the core of the manifold $M$. It is also connected. The core $\widetilde{M}$ can be viewed as a neighborhood of the intersection of the thick part and the convex core. 
\end{enumerate}

\section{Bounding Small Eigenvalues from Below} 
Recall that our goal is to understand the relationships between the eigenvalues of the manifold $M$ and eigenvalues of the core $\widetilde{M}$. Although the core is compact, topologically it is as complicated as $M$, making explicit computations on the core difficult. On the contrary, its complement has concrete and simpler geometry, each component of which is essentially a warped product. It consists of Margulis tubes, finite-volume cusps, and neighborhoods of geometrically finite ends of infinite volume. Each type of component gives rise to a functional inequality. We start by considering the components of ends. In general, the metric on an end, which is exponentially expanding, is quasi-isometric to $\cosh^2 t ds^2_{\partial C_1(M)} + dt^2$, where $ds^2_{\partial C_1(M)}$ is a metric on the boundary of $1$-neighborhood of the convex core (see \cite{canary1993ends} and \cite{lott1997l2}). The eigenvalues under quasi-isometry have certain stability behavior, which depends on the constant of the quasi-isometry (see comments under [\citenum{colbois2003extremal}, Lemma 2.2]).  

\begin{lemma}\label{lm3.1}
	Let $\overline{C_{r}(M)}$ denote the closure of the $r$-neighborhood of the convex core, $T$ denote a component of $M-\overline{C_{r}(M)}$, which is a neighborhood of an end. Let $f$ be a smooth function with compact support on T. Then 
\end{lemma}
\[
\displaystyle \int_T \| \nabla f \|^2 \geq \frac{(n-1)^2(\tanh r)^2}{4} \int_T f^2 . 
\] 
\begin{proof}
	By [\citenum{hamenstadt2004small}, Lemma 2.3], the infimum of Rayleigh quotients $\mathcal{R}(f)=\int_M \| \nabla f \|^2 /\int_M f^2$ of all such functions on $T$ is bounded below by $(n-1)^2 (\tanh r)^2/4$, which is equivalent to the desired inequality. 
\end{proof}

 Next, we consider the thin parts. For Margulis tubes and finite-volume cusps, we rewrite the concluding inequality of [\citenum{hamenstadt2018small}, Lemma 2.3] as \Cref{lm3.2} for the proof of \Cref{pr3.4}. 

\begin{lemma}\label{lm3.2}
	Suppose $T \subset M_{thin}$ is a Margulis tube or a cusp with boundary $\partial T$, and f is a smooth function on T with $\int_T f^2 \geq \int_{\partial T} f^2$. Then \par
	\begin{center}
		$ \displaystyle \int_T \| {\nabla f} \|^2 \geq   \frac{(n-2)^2}{4} \int_T f^2 $.
	\end{center}
\end{lemma} \par
The proof of the above lemma holds more generally in variable negative curvatures; it uses Jacobian comparison under variation of curvatures and integration by parts. The coefficient of \Cref{lm3.2} for cusps of finite volume can be improved. The following inequality is a special case of [\citenum{dodziuk1986lower}, Lemma 1]: 
\begin{lemma}\label{lm3.3}
	Let $T$ be a cusp corresponding to a parabolic subgroup $\Gamma$ of maximal rank $n-1$. Suppose $f$ is a smooth function with compact support in the interior of $T$. Then \par
	\begin{center}
		$ \displaystyle \int_T \| {\nabla f} \|^2 \geq   \frac{(n-1)^2}{4} \int_T f^2 $.
	\end{center}
\end{lemma}

The proof of \Cref{lm3.3} is a consequence of formula $(3)$ on page $3$ from \cite{mckean1970upper}. We only need to show similar inequality is true in the $t$ direction and then integrate over the base $H/ \Gamma$. The following proposition generalizes Lemma 2.4 of \cite{hamenstadt2018small} with essentially the same technique. Recall that the core $\widetilde{M} = \widetilde{M}(r)$ depends on $r$. The restriction on values of $r$ ensures that the estimates on the ends of infinite volume fit those of the thin part from \cite{hamenstadt2018small}.

\begin{proposition}\label{pr3.4}
	Let $f: M \rightarrow \mathbb{R}$ be a smooth, square-integrable function with Raleigh quotient $\mathcal{R}(f)< (n-2)^2 / 12$. Then \par
	\begin{center}
		$ \displaystyle \int_{\widetilde{M}} f^2 \geq  \frac{1}{3} \int_M f^2 $  ,
	\end{center}
	where $r\geq \tanh ^{-1} \frac{(n-2)^2}{(n-1)^2}$.
\end{proposition}

\begin{proof}
	Recall that a neighborhood of an end of infinite volume is a component of $M-\overline{C_{r}(M)}$. Let $M^*$ denote the disjoint union of all components of Margulis tubes, cusps, and neighborhoods of ends of infinite volume: $M^*=\cup _{i=1}^k T_i$. The union is finite for a geometrically finite manifold. For simplicity of notation, for each $T_i$ which belongs to the thin part, denote by $r_i$ the radial distance function to the boundary hypersurface, i.e., $r_i (x)=$ length of a radial arc connecting $x\in T$ to $\partial T$. For each $T_i$ which belongs to the ends of infinite volume, let $r_i (x)$ denote the distance to $\partial \widetilde{M}$, the boundary of the core. In each $T_i$, we consider the shell consisting of points whose distance to $\partial T_i$ is less than or equal to $1$. Denote by $A$ the finite disjoint union of all such shells. Recall that the thick core $\widetilde{M}$ is defined to be $(M-int(M^*)) \cup A$. Moreover, up to reordering, we may assume that there exists $p \leq k$ such that for $i\leq p$, there exist $s_i \leq 1$ such that  \par
	
	\begin{center}
		$\displaystyle \int_{\{r_i=s_i\} \cap T_i} f^2 \leq \int_{\{r_i\geq s_i \}\cap T_i} f^2$
	\end{center} 
	and that for $i>p$, such $s_i$'s do not exist. The volume element on the hypersurfaces ${r_i = s_i}$ is as in Lemma 2.3 of \cite{hamenstadt2018small} and Lemma 2.3 of \cite{hamenstadt2004small}. If $\int_{M^*-A}f^2 \leq \frac{2}{3}\int_M f^2$ we are done. Thus we assume $\sum_{i=1}^k \int_{T_i-A}f^2 = \int_{M^*}f^2 > \frac{2}{3}\int_M f^2$. There are two cases. In the first case, $\sum_{i=1}^p \int_{T_i-A}f^2 \geq \frac{1}{3}\int_M f^2 $. Suppose among the $p$ components, $p_1$ components are tubes and cusps, and the other $p-p_1$ components are neighborhoods of ends of infinite volume. Combining \Cref{lm3.1} and \Cref{lm3.2} shows that
	\begin{align*}
		\int_M \|{\nabla f}\|^2  
		& \geq \sum_{i=1}^{p} \int_{T_i \cap \{ r_i \geq s_i \}} \| {\nabla f} \|^2 \\
		& \geq \displaystyle \frac{(n-2)^2}{4} \sum_{i=1}^{p_1}  \int_{\{r_i\geq s_i \}\cap T_i} f^2 +\frac{(n-1)^2\tanh{r}}{4} \sum_{i=p_1}^{p}  \int_{\{r_i\geq s_i \}\cap T_i} f^2 \\
		& \geq \frac{(n-2)^2}{12}\int_M f^2,
	\end{align*}
	where we use the assumption on $r$ in the last step. This contradicts the assumption on the Rayleigh quotient of $f$. 
	In the second case, $\sum_{i=1}^p \int_{T_i-A} f^2 < \frac{1}{3}\int_M f^2$ implies that 
	\begin{center}
		$\displaystyle \sum_{i=p+1}^k \int_{T_i -A} f^2 \geq \frac{1}{3}\int_M f^2 $.
	\end{center}
	For each $i>p$, if we integrate the defining equation  
	\begin{center}
		$\displaystyle \int_{\{r_i=s_i\} \cap T_i} f^2 \geq \int_{ \{r_i\geq s_i \}\cap T_i} f^2$
	\end{center}
	over the shell $T_i \cap A = \{0 \leq r_i \leq 1 \}  $, we obtain
	\begin{align*}
		 \int_{A \cap T_i} f^2 
		 &=\int_0^1 ds \int_{\{r_i=s_i\} \cap T_i} f^2  \geq \int_0^1 ds \int_{ \{r_i\geq s_i \}\cap T_i} f^2 \\
		 &\geq \int_0^1 ds \int_{ \{r_i\geq 1 \}\cap T_i} f^2 = \int_{T_i -A} f^2.
	\end{align*}
	Summing over $i \geq p+1$ and using inequality, we obtain 
	\begin{center}
		$ \displaystyle \int_{\cup_{i=p+1}^k T_i \cap A} f^2 \geq \sum_{i=p+1}^k \int_{T_i-A }f^2 \geq \frac{1}{3}\int_M f^2 $.
	\end{center} 
	As $A \subset \widetilde{M}$, this contradicts the assumption on $f$. The lemma follows. 
\end{proof}

Now we are ready to complete the proof of our main theorem. Our argument is essentially the same as \cite{hamenstadt2018small} for the finite-volume case. We clarify the argument from \cite{hamenstadt2018small}, which demonstrates the nondegenercy of the operator that restricts a function on $M$ to its core $\wm$. 

\begin{theorem}\label{th3.5}
	Let M be an oriented geometrically finite hyperbolic manifold $M$ of dimension $n \geq 3$ of infinite volume. Then for all $k \geq 0$, we have 
	\begin{center}
		$\Large \lambda_k(M) \geq $ min \{$\frac{1}{3} \lambda_k(\widetilde{M}) ,\frac{(n-2)^2}{12}$\}.
	\end{center} \par
\end{theorem}

\begin{proof}
	Let $M$ be an oriented geometrically finite hyperbolic manifold $M$ of dimension $n \geq 3$ of infinite volume. By Theorem 2.12 of \cite{Hislop1994}, in the interval $[0, \frac{(n-2)^2}{12})$, $M$ has at most finitely many discrete eigenvalues, each of which has finite multiplicity. The idea of the proof is that the projection of the first $k$ eigenfunctions to some function space $\mathcal{H}(\widetilde{M})$ is nondegenerate and spans a $k$-dimensional subspace for $\mathcal{H}(\widetilde{M})$. Then applying Min-max theorem and \Cref{pr3.4} will give the desired result. Let $\mathcal{H}(M)$ (respectively $ \mathcal{H}(\widetilde{M})$) be the Sobolev space of square-integrable functions with square-integrable weak derivatives on $M$ ( $\widetilde{M}$ has smooth boundary  by construction). The class $\mathcal{H}(\widetilde{M})$ contains all functions that are smooth in the interior of, and, up to the boundary of $\widetilde{M}$(see page 14 to 17 in \cite{1984363}). Let $k>0$ be such that $\lambda_k(M)<(n-2)^2/12$. We construct a $k$-dimensional linear subspace of the Hilbert space $\mathcal{H}(M)$, which correspond to the direct sum of the first $k$ eigenspaces. If $\lambda_i$ is simple, denote by $E_i$ the eigenspace corresponding to $\lambda_i$. If $\lambda_i>\lambda_{i-1}$ and $\lambda_i$ has multiplicity $m_i$, let $E_i, \cdots, E_{i+m_i-1}$ be linearly independent $1$-dimensional eigenspace corresponding to $\lambda_i$. Note that our choice of the eigenspaces will not affect later arguments. Let $E$ denote the direct sum of first $k$ eigenspaces: $E\coloneqq \bigoplus_{j=1}^k E_j$. \par 
	In order to relate eigenvalues of $\widetilde{M}$ to those of $M$, we consider the projection or restriction map
	\begin{equation}
		\pi : \mathcal{H}(M) \rightarrow \mathcal{H}(\widetilde{M}).	
	\end{equation}
	Since smooth functions are dense in  $\mathcal{H}(M)$ and $\mathcal{H}(\widetilde{M})$, $\pi$ is a $1$-Lipschitz linear map. Denote $\pi (E) $ by $W$. We  will use \Cref{pr3.4} to show that $\pi$ is nondegenerate, i.e., the dimension of $W$ is $k$, which is the same as that of $E$. Our first step is to show that if there is a nontrivial linear combination of eigenfunctions whose restriction to the core is zero, then they must correspond to the same eigenvalue. We proceed by induction. We start with the case of two functions. To simplify notations, assume there is a nontrivial linear combination of $f_i$ and $f_j$ such that $c_1 f_i +c_2f_j=0$, where $f_i$ corresponds to $\lambda_i$ and has Rayleigh quotient $<(n-2)^2/12 $. By \Cref{pr3.4}, the constant $c_1$ and $c_2$ are nonzero. Apply the Laplacian operator to the equation, we get $c_1 \lambda_i f_i +c_2 \lambda_j f_j=0$. If $\lambda_i \neq \lambda_j$ we have $f_i=f_j=0$ on $\widetilde{M}$, contradicting Proposition \ref{pr3.4}. Now suppose that the statement is true for $k-1$ eigenfunctions and
	\[
	c_1f_{i_1} + \cdots +c_kf_{i_k}=0. 
	\]
	Without loss of generality, all coefficients $c_i$ are nonzero since otherwise, it follows from the induction hypothesis already. Applying the Laplacian operator to the above equation $k-1$ times, we get a linear system of equations whose coefficient matrix is 
	\[
	\begin{bmatrix}
		c_1 &  \dots &c_{k} \\
		c_1 \lambda_{i_1} & \dots & c_k \lambda_{i_k} \\
		\vdots & \ddots & \\
		c_1 \lambda_{i_1}^{k-1} & \dots & c_k \lambda_{i_k}^{k-1}
	\end{bmatrix}.
	\]
	The determinant of this matrix is $c_1\cdots c_k \Pi_{1\leq p < q \leq k} (\lambda_{i_q}-\lambda_{i_p})$ by the properties of Vandermonde matrix. If all eigenvalues are distinct, then the nonzero determinant forces all $f_{i_p}$ to vanish in the core, contradicting \Cref{pr3.4}. This finishes the induction step. \par 
	If there is a function $f=c_1 f_{i_1} + ... +c_k f_{i_k} \in E$ such that $\pi(f)=0$, it then follows that all $f_i$ must correspond to the same eigenvalue $\lambda_i < (n-2)^2/12$. Thus $f$ is an eigenfunction corresponding to $\lambda_i$ whose restriction to $\wm$ is zero. This is again a contradiction. Therefore we have proven that $\pi$ is nondegenerate, and hence $dim(W)=dim(E)=k$. Let $f$ belong to $E$, the direct sum of the first $k$ eigenspaces. From properties of the Rayleigh quotient, we have $\mathcal{R}(f) \leq \lambda_k (M)$, which is equal to the largest Rayleigh quotient of eigenfunctions in $E$. Note that both $f$ and $\pi (f)$ are smooth. Again, using 
	\begin{center}
		$\displaystyle \int_{\widetilde{M}} f^2 \geq  \displaystyle \frac{1}{3}$,\par
	\end{center}
	we have \par
	\begin{center}
		\large{$\lambda_k(M)  = \underset{f_i \in E}{\sup} $  $\mathcal{R}(f_i) =\underset{f_i \in E}{\sup}  \displaystyle \int_{M} \| {\nabla f_i} \| ^2 \geq \underset{f_i \in E}{\sup} \int_{\widetilde{M}} \| {\nabla f_i} \| ^2$} \par
		
		\large{$= \underset{f_i \in E}{\sup}$ $\displaystyle \mathcal{R}(f_i|_{\widetilde{M}}) \int_{\widetilde{M}} f^2 \geq \underset{f_i \in E}{\sup}$ $\displaystyle \frac{1}{3} \mathcal{R}(f_i|_{\widetilde{M}}) \geq \frac{1}{3} \lambda_k (\widetilde{M})$ }. \par
	\end{center}
	The last inequality follows from two facts. First, note that the restriction $\pi(f_i)=f_i|_{\widetilde{M}}$ of $f_i$ a priori satisfies no boundary condition. Thus the min-max characterization using $f_i|_{\widetilde{M}}$ gives rise to Neumann eigenvalues of $\widetilde{M}$. The second fact is that $W$ is only some $k$-dimensional subspace, not necessarily the one spanned by the first $k$ eigenfunctions of $\widetilde{M}$.
\end{proof}

\begin{corollary}
	There exists a constant $C(n)$ that only depends on the dimension such that 
	\begin{equation}
		\lambda_0(M) \geq \frac{C(n)}{\vol(\widetilde{M})^2}. 
	\end{equation}
\end{corollary}
\begin{proof}
	The argument is similar to the proof of [\citenum{hamenstadt2018small}, Corollary 2.6], where the thick part $M_{thick}$ is compact and uniformly quasi-isometric to a finite connected graph whose first nontrivial eigenvalue is uniformly comparable to $\lambda_1(M_{thick})$. Note that for geometrically finite $M$ of infinite volume, the thick part $M_{thick}$ is not compact since it contains neighborhoods of ends. In constructing the core $\widetilde{M}$, we intersect the thick part with a neighborhood of the convex core,  which results in a compact set. This intersection has a lower bound on the injectivity radius, which depends on the Margulis constant $\epsilon(n)$. It remains to show that a lower bound on the injectivity radius, which depends only on the dimension, remains after finitely many shells are glued to the intersection. For a Margulis tube $T$, Buser, Colbois, and Dodziuk [\citenum{buser1993tubes}, Lemma 2.4] show that the radius $R_T$ of a tube $T$ grow towards infinity if the length $\ell$ of the core geodesic approaches $0$ (note that their result generalizes a theorem by Meyerhoff on [\citenum{mrLowerVolHyp}, P1042]). The proof of [\citenum{buser1993tubes}, Corollary 2.24] can be adapted to argue that the injectivity radius of the shell is uniformly bounded below by a constant depending on the dimension. For the cusp case, since the injectivity radius of its boundary is approximately $\epsilon(n)$, using the metric of the cusp (Section 2, \ref{cusp}), one can show the shells also have a uniform lower bound on the injectivity radius. Thus the core $\widetilde{M}$ is uniformly quasi-isometric to a finite connected graph $G$, whose number of vertices is approximately $\vol(\widetilde{M})$. 
\end{proof}

\section{Future Directions}
In dimension three, for finite-volume oriented hyperbolic $3$-manifolds $M$, Hamenst\"{a}dt \cite{hamenstadt2018small} proves
\begin{center}
	$\lambda_k(M) \leq c \log(vol(M_{thin})+2)\lambda_k(M_{thick}) $
\end{center}
for all $k \geq 1$ such that $\lambda_k(M_{thick}) < 1/96$. Her technique builds on special properties of $3$-dimensional Margulis tubes and does not generalize naturally to higher dimensions. Canary [\citenum{canary1992}, Theorem A] obtains an upper bound on the infimum of the spectrum of geometrically finite manifolds of infinite volume in dimension $3$: 
\[
\lambda_{0}(M) \leq c\frac{|\chi(\partial C(M))|}{\vol (C(M))},
\]
where $\chi(\partial C(M))$ is the Euler characteristic of the boundary surface of the convex core. 
It is natural to ask for generalization of results of either Hamenst\"{a}dt or Canary to higher dimensions or general eigenvalues. Moreover, many questions remain uninvestigated. What is the pattern of the distribution of the eigenvalues of noncompact negatively curved manifolds? Do they distribute relatively uniformly, or can they cluster around one point, like zero or the bottom of the essential spectrum? The interaction between eigenvalues and the Gromov-Hausdorff convergence of manifolds is also interesting. By [\citenum{canary1992}, Theorem A], a geometrically infinite, doubly degenerate hyperbolic $3$-manifold $M$ has the bottom $\lambda_0$ of the spectrum equal to $0$. Using Thurston's double limit theorem [\citenum{wtHypFiberII}, Theorem 4.1] (see also [\citenum{canary1992}, Remark (2) on P365]), one can construct a sequence of geometrically finite, quasi-Fuchsian hyperbolic $3$-manifolds $M_j$ such that the volume of the convex core $\vol(C(M_j))$ tends to infinity and $M_j \rightarrow M$ geometrically (geometric convergence of hyperbolic manifolds is equivalent to the Gromov-Hausdorff convergence). Thus $\lambda_0(M_j)$ tends to $0$ which is $\lambda_0(M)$. It implies that $\lambda_0$ might satisfy certain continuity with respect to geometric convergence, which is confirmed in [\citenum{ctHausdorffDimLimit}, Theorem 4.1]. When a sequence of closed hyperbolic $3$-manifolds $M_j$ converge geometrically to a noncompact finite-volume manifold $M$, I. Chavel and J. Dodziuk \cite{cdSpectrumDege} have investigated how fast the discrete eigenvalues of $M_j$ cluster around the continuous spectrum of $M$. A natural question is whether the higher eigenvalues of infinite-volume $M_j$ satisfy similar continuity with respect to convergence. Recently, B. Liu and F. Pallete \cite{lbpfUniformSpectralGap} have considered similar problems for the spectral gap of geometrically finite manifolds. \par  
Moreover, eigenvalues for differential forms show up unexpectedly as a bridge to connect Floer homology on 3-manifolds and hyperbolic geometry. See \cite{lin2017monopole} and the reference therein. While monopole Floer homology is difficult to compute directly, using eigenvalues of forms to mediate between geometry and Floer homology seems practical. It will be interesting to develop a stronger connection between such eigenvalues and the underlying hyperbolic geometry.

\section*{Acknowledgment}
The author would like to express deep gratitude to his advisor Nathan Dunfield for continuing support, helpful insights and comments through the years of his Ph.D. study. The author thanks Pierre Albin, John D'Angelo, Anil Hirani, and Hadrian Quan for helpful discussions and emails. Special thanks to Richard Laugesen for clarification in the proof of Theorem 3.5. We are also very grateful to the anonymous reviewers for their careful reading of our manuscript and many insightful comments which improve the paper. This work is partially supported by NSF grant DMS-1811156.

\bibliographystyle{plain}
\bibliography{References}

Xiaolong Hans Han\footnote{Current Address: Yau Mathematical Sciences Center, Tsinghua University, Beijing, Beijing 100084, China. xlhan@mail.tsinghua.edu.cn}\\
Email: \\
\href{mailto:xhan25@illinois.edu}{xhan25@illinois.edu} \\
Address: \\
273 Altgeld Hall
1409 W. Green Street 
Urbana, IL 61801

\end{document}